\pdfoutput=1
\documentclass[11pt]{article}

\usepackage[utf8]{inputenc}
\usepackage[pdftex]{graphicx,xcolor}
\usepackage{hyperref}
\hypersetup{pdftex,colorlinks=true,allcolors=blue}
\usepackage{array,bookmark,fullpage}
\usepackage[nottoc]{tocbibind}
\usepackage{amsmath,amssymb,amsthm,mathtools,bm,tikz-cd,stmaryrd}
\usepackage[capitalize]{cleveref}

\usepackage{mymacros}

\let\defn\textbf

\begin{document}

\title{Borel structurability by locally finite simplicial complexes}
\author{Ruiyuan Chen\thanks{Research partially supported by NSERC PGS D}}
\date{}
\maketitle

\begin{abstract}
We show that every countable Borel equivalence relation structurable by $n$-dimensional contractible simplicial complexes embeds into one which is structurable by such complexes with the further property that each vertex belongs to at most $M_n := 2^{n-1}(n^2+3n+2)-2$ edges; this generalizes a result of Jackson-Kechris-Louveau in the case $n = 1$.  The proof is based on that of a classical result of Whitehead on countable CW-complexes.
\end{abstract}

\section{Introduction}
\label{sec:intro}

A \defn{countable Borel equivalence relation} $E$ on a standard Borel space $X$ is a Borel equivalence relation $E \subseteq X^2$ for which each equivalence class is countable.  The class of \defn{treeable} countable Borel equivalence relations, for which there is a Borel way to put a tree (acyclic connected graph) on each equivalence class, has been studied extensively by many authors, especially in relation to ergodic theory; see e.g., \cite{Ada}, \cite{Ga1}, \cite{JKL}, \cite{KM}, \cite{HK}, \cite{Hjo}.  It is a basic result, due to Jackson-Kechris-Louveau \cite[3.10]{JKL}, that every treeable equivalence relation embeds into one treeable by trees in which each vertex has degree at most 3.  The purpose of this paper is to present a generalization of this result to higher dimensions.

Recall that a \defn{simplicial complex} on a set $X$ is a collection $S$ of finite nonempty subsets of $X$ which contains all singletons and is closed under nonempty subsets.
A simplicial complex $S$ has a \defn{geometric realization} $\abs{S}$, which is a topological space formed by gluing together Euclidean simplices according to $S$ (see \cref{sec:prelims} for the precise definition); $S$ is \defn{contractible} if $\abs{S}$ is.
Given a distinguished class $\@K$ of simplicial complexes (e.g., the contractible ones) and a countable Borel equivalence relation $(X, E)$, a \defn{(Borel) structuring of $E$ by simplicial complexes in $\@K$} is, informally (see \cref{sec:prelims}), a Borel assignment of a simplicial complex $S_C \in \@K$ on each equivalence class $C \in X/E$.  If such a structuring exists, we say that $E$ is \defn{structurable by complexes in $\@K$}.
We are interested here mainly in $\@K =$ $n$-dimensional contractible simplicial complexes; when $n = 1$, we recover the notion of treeability.  The study of equivalence relations structurable by $n$-dimensional contractible simplicial complexes was initiated by Gaboriau \cite{Ga2}, who proved (among other things) that for $n = 1, 2, 3, \dotsc$ these classes of countable Borel equivalence relations form a strictly increasing hierarchy under $\subseteq$.

Recall also the notion of a \defn{Borel embedding} $f : E -> F$ between countable Borel equivalence relations $(X, E)$ and $(Y, F)$, which is an injective Borel map $f : X -> Y$ such that $x \mathrel{E} x' \iff f(x) \mathrel{F} f(x')$ for all $x, x' \in X$.

\begin{theorem}
\label{thm:lfsc-cber}
Let $n \ge 1$, and let $(X, E)$ be a countable Borel equivalence relation structurable by $n$-dimensional contractible simplicial complexes.  Then $E$ Borel embeds into a countable Borel equivalence relation $(Y, F)$ structurable by $n$-dimensional contractible simplicial complexes in which each vertex belongs to at most (or even exactly) $M_n := 2^{n-1}(n^2+3n+2)-2$ edges.
\end{theorem}

In particular, every $E$ structurable by $n$-dimensional contractible simplicial complexes Borel embeds into an $F$ structurable by locally finite such complexes, where a simplicial complex is \defn{locally finite} if each vertex is contained in finitely many edges (or equivalently finitely many simplices).  The constant $M_n$ above is not optimal: for $n = 1$ we have $M_1 = 4$, whereas by the aforementioned result of Jackson-Kechris-Louveau we may take $M_1 = 3$ instead, which is optimal; for $n = 2$ we have $M_2 = 22$, whereas by a construction different from the one below we are able to get $M_2 = 10$.  We do not know what the optimal $M_n$ is for $n > 1$; however, the result of Gaboriau mentioned above implies that the optimal $M_n$ is at least $n+1$.

The referee has pointed out that by an easy argument, one may strengthen ``at most'' to ``exactly'' in \cref{thm:lfsc-cber} (as well as in the following reformulations).

We may reformulate \cref{thm:lfsc-cber} in terms of \defn{compressible} countable Borel equivalence relations, which are those admitting no invariant probability Borel measure (see e.g., \cite{DJK} for various equivalent definitions of compressibility):

\begin{corollary}
\label{thm:lfsc-compr}
Let $n \ge 1$, and let $(X, E)$ be a compressible countable Borel equivalence relation structurable by $n$-dimensional contractible simplicial complexes.  Then $E$ is structurable by $n$-dimensional contractible simplicial complexes in which each vertex belongs to at most (or even exactly) $M_n$ edges.
\end{corollary}

Note that by the theory of cost (see \cite{Ga1}, \cite{KM}), \cref{thm:lfsc-compr} cannot be true of non-compressible equivalence relations, i.e., there cannot be a uniform bound $M_n$ on the number of edges containing each vertex.

\Cref{thm:lfsc-cber} fits into a general framework for classifying countable Borel equivalence relations according to the (first-order) structures one may assign in a Borel way to each equivalence class; see \cite{JKL}, \cite{Mks}, \cite{CK}.  As with most such results, the ``underlying'' result is that there is a procedure for turning every structure of the kind we are starting with ($n$-dimensional contractible simplicial complexes) into a structure of the kind we want ($n$-dimensional contractible simplicial complexes satisfying the additional condition), which is ``uniform'' enough that it may be performed simultaneously on all equivalence classes in a Borel way.  We state this as follows.  We say that a simplicial complex is \defn{locally countable} if each vertex is contained in countably many edges (or equivalently countably many simplices).

\begin{theorem}
\label{thm:lfsc-simp}
There is a procedure for turning a locally countable simplicial complex $(X, S)$ into a locally finite simplicial complex $(Y, T)$, such that
\begin{itemize}
\item[(i)]  $T$ is homotopy equivalent to $S$;
\item[(ii)]  if $S$ is $n$-dimensional, then $T$ can be chosen to be $n$-dimensional and with each vertex in at most (or even exactly) $M_n$ edges.
\end{itemize}
Furthermore, given a countable Borel equivalence relation $(X, E)$ and a structuring $S$ of $E$ by simplicial complexes, this procedure may be performed simultaneously (in a Borel way) on all $E$-classes, yielding a countable Borel equivalence relation $(Y, F)$ with a structuring $T$ by simplicial complexes and a Borel embedding $f : E -> F$ such that applying the above procedure to the complex $S_{[x]_E}$ on an $E$-class $[x]_E$ yields the complex $T_{[f(x)]_F}$ on the corresponding $F$-class $[f(x)]_F$.
\end{theorem}

The theorem in this form also yields the following (easy) corollary:

\begin{corollary}
\label{cor:lfsc-cber}
Every countable Borel equivalence relation $(X, E)$ embeds into a countable Borel equivalence relation $(Y, F)$ structurable by locally finite contractible simplicial complexes.
\end{corollary}

Again, this may be reformulated as

\begin{corollary}
\label{cor:lfsc-compr}
Every compressible countable Borel equivalence relation $(X, E)$ is structurable by locally finite contractible simplicial complexes.
\end{corollary}

The proof of \cref{thm:lfsc-simp} is based on a classical theorem of Whitehead on CW-complexes \cite[Theorem~13]{Wh}, which states that every locally countable CW-complex is homotopy equivalent to a locally finite CW-complex of the same dimension.  While the statement of this theorem is useless for \cref{thm:lfsc-simp} (every contractible complex is homotopy equivalent to a point, but one cannot replace every class of a non-smooth equivalence relation with a point), its proof may be adapted to our setting, with the help of some lemmas from descriptive set theory.

We review some definitions and standard lemmas in \cref{sec:prelims}, then give the proofs of the above results in \cref{sec:proofs}; the proofs are structured so that it should be possible to read the combinatorial/homotopy-theoretic argument without the descriptive set theory, and vice-versa.  In \cref{sec:future} we list some other properties of treeable equivalence relations which we do not currently know how to generalize to higher dimensions.

\medskip
\textit{Acknowledgments.}  We would like to thank Alexander Kechris, Damien Gaboriau, and the anonymous referee for providing some comments on drafts of this paper.

\section{Preliminaries}
\label{sec:prelims}

We begin by reviewing some notions related to simplicial complexes; see e.g., \cite{Spa}.

A \defn{simplicial complex} on a set $X$ is a set $S$ of finite nonempty subsets of $X$ such that $\{x\} \in S$ for all $x \in X$ and every nonempty subset of an element of $S$ is in $S$.  The elements $s \in S$ are called \defn{simplices}.  The \defn{dimension} $\dim(s)$ of $s \in S$ is $|s|-1$; if $\dim(s) = n$, we call $s$ an \defn{$n$-simplex}.  We let $S^{(n)} := \{s \in S \mid \dim(s) = n\}$ be the $n$-simplices, and call $S$ \defn{$n$-dimensional} if $S^{(m)} = \emptyset$ for $m > n$.  (To avoid confusion, we will sometimes call a simplicial complex with an $n$-simplex containing all other simplices a \defn{standard $n$-simplex}.)

A \defn{subcomplex} of $(X, S)$ is a simplicial complex $(Y, T)$ such that $Y \subseteq X$ and $T \subseteq S$.  For a simplicial complex $(X, S)$ and a subset $Y \subseteq X$, the \defn{induced subcomplex} on $Y$ is $S|Y := \{s \in S \mid s \subseteq Y\}$.  A \defn{simplicial map} $f : S -> T$ between complexes $(X, S)$ and $(Y, T)$ is a map $f : X -> Y$ such that $f(s) \in T$ for all $s \in S$.

The \defn{geometric realization} of a simplicial complex $(X, S)$ is the topological space $|S|$ formed by gluing together standard Euclidean $n$-simplices $\Delta^n$ for each $s \in S^{(n)}$, according to the subset relation.  Explicitly, $|S|$ can be defined as the set $\bigcup_{s \in S} |s|_S \subseteq [0, 1]^X$, where $|s|_S := \{(a_x)_{x \in X} \mid \sum_{x \in X} a_x = 1,\, \forall x \not\in s\, (a_x = 0)\}$ is (thought of as) the set of formal convex combinations of elements of $X$ supported on $s$, equipped with the topology where a subset of $|S|$ is open iff its intersection with each $|s|_S$ is open in the Euclidean topology on $|s|_S$.  We say that $S$ is \defn{contractible} if $|S|$ is.  Likewise, a simplicial map $f : S -> T$ induces a continuous map $|f| : |S| -> |T|$ in the obvious way; we say that $f$ is a \defn{homotopy equivalence} if $|f|$ is.

We also need the more refined notion of an \defn{ordered simplicial complex}, which is a simplicial complex $S$ on a poset $X$ such that every simplex $s \in S$ is a chain $\{x_0 < \dotsb < x_n\}$ in $X$.  The \defn{product} of ordered simplicial complexes $(X, S)$ and $(Y, T)$ is the complex $(X \times Y, S \times T)$ where $X \times Y$ is the usual product poset and
\begin{align*}
\{(x_0, y_0) \le \dotsb \le (x_n, y_n)\} \in S \times T \iff \{x_0 \le \dotsb \le x_n\} \in S \AND \{y_0 \le \dotsb \le y_n\} \in T.
\end{align*}
It is standard that $|S \times T|$ is canonically homeomorphic to $|S| \times |T|$ with the CW-product topology (which coincides with the product topology if $S, T$ are locally countable).

In order to prove contractibility/homotopy equivalence, we use the following standard results from homotopy theory.

\begin{lemma}
\label{lm:cech-nerve}
Let $S, T$ be simplicial complexes which are the unions of subcomplexes $S = \bigcup_{i \in I} S_i$ and $T = \bigcup_{i \in I} T_i$ over the same index set $I$, and let $f : S -> T$ be a simplicial map such that $f(S_i) \subseteq T_i$ for each $i$.  If for each finite family of indices $i_1, \dotsc, i_n \in I$, the restriction $f : S_{i_1} \cap \dotsb \cap S_{i_n} -> T_{i_1} \cap \dotsb \cap T_{i_n}$ is a homotopy equivalence, then $f : S -> T$ is a homotopy equivalence.
\end{lemma}
\begin{proof}
See e.g., \cite[4K.2]{Hat}.
\end{proof}

\begin{corollary}
\label{lm:pushout}
Let $S$ be a simplicial complex which is the union of subcomplexes $U, V \subseteq S$.  If the inclusion $U \cap V -> U$ is a homotopy equivalence, then so is the inclusion $V -> S$.  In particular, if $U$, $V$, and $U \cap V$ are contractible, then so is $S$.
\end{corollary}
\begin{proof}
Apply \cref{lm:cech-nerve} to the inclusion from $V = (U \cap V) \cup V$ into $S = U \cup V$.
\end{proof}

\begin{corollary}
\label{lm:dirunion}
Let $S = \bigcup_{i \in I} S_i$ and $T = \bigcup_{i \in I} T_i$ be simplicial complexes which are directed unions of subcomplexes (over the same directed poset), and let $f : S -> T$ be a simplicial map such that $f(S_i) \subseteq T_i$ for each $i$.  If each restriction $f|S_i : S_i -> T_i$ is a homotopy equivalence, then so is $f$.

In particular, if $S_i$ is contractible for each $i$, then (taking $T = T_i = $ a point) $S$ is contractible.
\end{corollary}
\begin{proof}
In the case where $I$ is a well-ordered set, this is immediate from \cref{lm:cech-nerve}; the two places below where we use this result both follow from this case.  (To deduce the general form of the result, one can appeal to Iwamura's lemma from order theory which reduces an arbitrary directed union to iterated well-ordered unions; see e.g., \cite{Mky}.)
\end{proof}

We say that a simplicial map $f : S -> T$ is a \defn{trivial pseudofibration} if for each $t \in T$, the subcomplex $S|f^{-1}(t) \subseteq S$ is contractible.

\begin{corollary}
\label{lm:quillena}
A trivial pseudofibration is a homotopy equivalence.
\end{corollary}
\begin{proof}
Apply \cref{lm:cech-nerve} to $S = \bigcup_{t \in T} S|f^{-1}(t)$ and $T = \bigcup_{t \in T} T|t$.
\end{proof}

Finally, we come to the notion of Borel structurability.  Let $(X, E)$ be a countable Borel equivalence relation.  We say that a simplicial complex $S$ on $X$ is \defn{Borel} if for each $n$ the $(n+1)$-ary relation ``$\{x_0, \dotsc, x_n\} \in S$'' is Borel, or equivalently $S$ is Borel as a subset of the standard Borel space of finite subsets of $X$.  A Borel simplicial complex $S$ on $X$ is a \defn{Borel structuring of $E$ by simplicial complexes} if in addition each simplex $s \in S$ is contained in a single $E$-class; such an $S$ represents the ``Borel assignment'' $C |-> S_C := S|C$ of the (countable) complex $S_C$ to each $E$-class $C \in X/E$.  More generally, for a class $\@K$ of simplicial complexes (e.g., the contractible ones), $S$ is a \defn{structuring of $E$ by complexes in $\@K$} if $S_C \in \@K$ for each $C \in X/E$; if such a structuring exists, we say that $E$ is \defn{structurable by complexes in $\@K$}.

\section{Proofs}
\label{sec:proofs}

\subsection{Some lemmas}

Let $N = \{\{i\}, \{i, i+1\} \mid i \in \#N\}$ denote the ordered simplicial complex on $\#N = \{0 < 1 < 2 < \dotsc\}$ with an edge between $i, i+1$ for each $i$, whose geometric realization is a ray.

For a simplicial complex $(X, S)$, a set $Y$, and a map $f : X -> Y$, define the \defn{image complex}
\begin{align*}
f(S) := \{f(s) \mid s \in S\},
\end{align*}
which is a simplicial complex on $f(X)$; we write $f(X, S)$ for $(f(X), f(S))$.
%(Indeed, if $\emptyset \ne t \subseteq f(s) \in f(S)$ then $t = f(s \cap f^{-1}(t))$.)
If $(X, S)$ is an ordered simplicial complex, $Y$ is a poset, and $f$ is monotone, then $(f(X), f(S))$ is also ordered.

Let $X$ be a poset and $T$ be an ordered simplicial complex on $X \times \#N^n$, for some $n \in \#N$.  We define the \defn{telescope} $\@T_n(T)$, an ordered simplicial complex on $X \times \#N^n$, by induction on $n$ as follows:
\begin{align*}
\@T_0(T) &:= T, \\
\@T_n(T) &:= (p_1(T) \times N) \cup (\@T_{n-1}(p_1(T)) \times \{0\}) \qquad\text{for $n \ge 1$},
\end{align*}
where $p_i : X \times \#N^n -> X \times \#N^{n-i}$ is the projection onto all but the last $i$ factors.  Explicitly, we have
\begin{align*}
\@T_n(T) = (p_1(T) \times N) \cup (p_2(T) \times N \times \{0\}) \cup \dotsb \cup (p_n(T) \times N \times \{0\}^{n-1}) \cup (p_n(T) \times \{0\}^n)
\end{align*}
(the last term $p_n(T) \times \{0\}^n$ is redundant unless $n = 0$).  Here are some simple properties of $\@T_n(T)$:

\begin{lemma}
\label{lm:telescope-props}
\begin{enumerate}
\item[(a)]  $T \subseteq \@T_n(T)$.
\item[(b)]  The projection $p_n : \@T_n(T) -> p_n(T)$ is a homotopy equivalence (with homotopy inverse the inclusion $p_n(T) \cong p_n(T) \times \{0\}^n \subseteq \@T_n(T)$).
\item[(c)]  For a subset $Z \subseteq X$, we have $\@T_n(T)|(Z \times \#N^n) = \@T_n(T|(Z \times \#N^n))$.
\item[(d)]  If $T$ is (at most) $k$-dimensional, then $\@T_n(T)$ is (at most) $(k+1)$-dimensional.
\end{enumerate}
\end{lemma}
\begin{proof}
(a), (c), and (d) are straightforward.  For $n \ge 1$, it is easily seen that $|\@T_n(T)|$ deformation retracts onto $|\@T_{n-1}(p_1(T)) \times \{0\}| \cong |\@T_{n-1}(p_1(T))|$; a simple induction then yields (b).
\end{proof}

We need one more (straightforward) lemma:

\begin{lemma}
\label{lm:simpcont-surj}
A trivial pseudofibration $f : S -> T$ is surjective on simplices.
\end{lemma}
\begin{proof}
Let $t \in T$.  Put $S' := \{s \in S \mid f(s) \subsetneq t\} = S|f^{-1}(t) \setminus \{s \in S \mid f(s) = t\}$.  Since $f$ is a trivial pseudofibration, for every $t' \subsetneq t$, $S'|f^{-1}(t') = S|f^{-1}(t')$ is contractible; thus $f : S' -> T|t \setminus \{t\}$ is a homotopy equivalence.  But $T|t \setminus \{t\}$ is the boundary of the simplex $t$, hence not contractible; thus for $S|f^{-1}(t)$ to be contractible, there must be $s \in S$ with $f(s) = t$.
\end{proof}

\subsection{The main construction}

We now give the main construction in the proof of \cref{thm:lfsc-simp}.  Let $(X, S)$ be a locally countable simplicial complex, which we may assume to be ordered by taking any linear order on $X$.  By local countability, for each $n$ we may find a function $c_n : S^{(n)} -> \#N$ which colors the intersection graph on the $n$-simplices $S^{(n)}$, which means that for $s, t \in S^{(n)}$ with $s \ne t$ and $s \cap t \ne \emptyset$ we have $c_n(s) \ne c_n(t)$.
%For example, we may simply enumerate the $n$-simplices in each connected component, then assign color $k$ to the $k$th simplices.
The idea is that for each $n$, we will multiply the complex by the ray $N$ and then attach each $n$-simplex $s \in S^{(n)}$ at position $c_n(s)$ along the ray, so that distinct simplices have non-overlapping boundaries.

Let $S_n := \bigcup_{m \le n} S^{(m)} = \{s \in S \mid \dim(s) \le n\}$, the $n$-skeleton of $S$.  We will inductively define ordered simplicial complexes $T_n$ on $X \times \#N^n$ and for $n \ge 1$, $T_n'$ on $X \times \#N^n$ such that
\begin{align*}
T_n \subseteq S_n \times N^n, &&
T_{n+1}' \subseteq S_n \times N^{n+1}, &&
T_n \times N \subseteq T_{n+1}' \subseteq T_{n+1},
\end{align*}
fitting into the following commutative diagram of monotone simplicial maps:
\begin{equation*}
\begin{tikzcd}
&&&& T_2 \times N \dar[->>,"p_1"',"\simeq"] \rar[hook] & T_3' \ar[dddl,->>,"p_3","\simeq"'] \rar[hook] & \dotsb \\
&& T_1 \times N \dar[->>,"p_1"',"\simeq"] \rar[hook] & T_2' \ar[ddl,->>,"p_2","\simeq"'] \rar[hook] & T_2 \ar[dd,->>,"p_2"',"\simeq"] \\
T_0 \times N \dar[->>,"p_1"',"\simeq"] \rar[hook] & T_1' \ar[dl,->>,"p_1","\simeq"'] \rar[hook] & T_1 \ar[d,->>,"p_1"',"\simeq"] \\
T_0 = S_0 \ar[rr,hook] && S_1 \ar[rr,hook] && S_2 \ar[rr,hook] && \dotsb
\end{tikzcd}
\tag{$*$}
\end{equation*}
The horizontal maps are the inclusions, while the vertical/diagonal maps are the projections $p_i : X \times \#N^n -> X \times \#N^{n-i}$ onto all but the last $i$ factors as before; furthermore each vertical/diagonal map will be a trivial pseudofibration between the respective complexes.

Start with $T_0 := S_0$.  Given $T_n$ such that $p_n : T_n -> S_n$ is a trivial pseudofibration, put
\begin{align*}
T_{n+1}' := (T_n \times N) \cup \bigcup_{s \in S^{(n+1)}} (\@T_n(T_n|(s \times \#N^n)) \times \{c_{n+1}(s)\}).
\end{align*}
Clearly this is an ordered simplicial complex on $X \times \#N^{n+1}$.

\begin{claim}
$p_{n+1} : (X \times \#N^{n+1}, T_{n+1}') -> (X, S_n)$ is a trivial pseudofibration.
\end{claim}
\begin{proof}
Let $t \in S_n$; we must check that $T_{n+1}'|p_{n+1}^{-1}(t) = T_{n+1}'|(t \times \#N^{n+1})$ is contractible.  We have
\begin{align*}
T_{n+1}'|(t \times \#N^{n+1})
&= (T_n|(t \times \#N^n) \times N) \cup \bigcup_{s \in S^{(n+1)}} (\@T_n(T_n|((s \cap t) \times \#N^n)) \times \{c_{n+1}(s)\}) \\
&= (\underbrace{T_n|p_n^{-1}(t) \times N}_A) \cup \bigcup_{s \in S^{(n+1)}} (\underbrace{\@T_n(T_n|p_n^{-1}(s \cap t)) \times \{c_{n+1}(s)\}}_{B_s})
\end{align*}
(using \cref{lm:telescope-props}(c)); let $A, B_s$ be as shown.  The subcomplex $A$ is contractible since $p_n : T_n -> S_n$ is a trivial pseudofibration by the induction hypothesis whence $T_n|p_n^{-1}(t)$ is contractible.  For each $s \in S^{(n+1)}$ such that $s \cap t \ne \emptyset$ (otherwise $B_s$ is empty), the subcomplex $B_s$ is contractible since the telescope $\@T_n(T_n|p_n^{-1}(s \cap t))$ is homotopy equivalent (by \cref{lm:telescope-props}(b)) to the projection $p_n(T_n|p_n^{-1}(s \cap t)) = p_n(T_n)|(s \cap t) = S_n|(s \cap t)$ which is a standard simplex; and also $A \cap B_s$ is contractible since
\begin{align*}
A \cap B_s
&= (T_n|(t \times \#N^n) \cap \@T_n(T_n|((s \cap t) \times \#N^n))) \times \{c_{n+1}(s)\} \\
&= (T_n|((s \cap t) \times \#N^n) \cap \@T_n(T_n|((s \cap t) \times \#N^n))) \times \{c_{n+1}(s)\} \\
&= T_n|((s \cap t) \times \#N^n) \times \{c_{n+1}(s)\} \\
&= T_n|p_n^{-1}(s \cap t) \times \{c_{n+1}(s)\}
\end{align*}
(the second equality since the telescope is a complex on $(s \cap t) \times \#N^n$, the third equality by \cref{lm:telescope-props}(a)), which is contractible because again $p_n$ is a trivial pseudofibration.  For two distinct $s, s' \in S^{(n+1)}$, we have $B_s \cap B_{s'} = \emptyset$: either $c_{n+1}(s) \ne c_{n+1}(s')$ in which case clearly $B_s \cap B_{s'} = \emptyset$, or $c_{n+1}(s) = c_{n+1}(s')$ whence by the coloring property of $c_{n+1}$ we have $s \cap s' = \emptyset$.  Now by repeated use of \cref{lm:pushout}, we get that $A \cup B_{s_1} \cup \dotsb \cup B_{s_m}$ is contractible for every finite collection of $s_1, \dotsc, s_m \in S^{(n+1)}$, whence by \cref{lm:dirunion}, $T_{n+1}'|(t \times \#N^{n+1})$ is contractible.
\end{proof}

Now put
\begin{align*}
T_{n+1} := T_{n+1}' \cup \{s \times \{0\}^n \times \{c_{n+1}(s)\} \mid s \in S^{(n+1)}\}.
\end{align*}

\begin{claim}
$T_{n+1}$ is an ordered simplicial complex on $X \times \#N^{n+1}$.
\end{claim}
\begin{proof}
The only thing that needs to be checked is that for each $s \in S^{(n+1)}$, a nonempty subset $s' \times \{0\}^n \times \{c_{n+1}(s)\}$ of $s \times \{0\}^n \times \{c_{n+1}(s)\}$ is still in $T_{n+1}$.  We may assume $s' \subsetneq s$.  Then $s' \in S_n$, so since $p_n : T_n -> S_n$ is a trivial pseudofibration, hence surjective on simplices, we have $s' \in p_n(T_n|(s \times \#N^n))$, whence $s' \times \{0\}^n \times \{c_{n+1}(s)\} \in p_n(T_n|(s \times \#N^n)) \times \{0\}^n \times \{c_{n+1}(s)\} \subseteq \@T_n(T_n|(s \times \#N^n)) \times \{c_{n+1}(s)\} \subseteq T_{n+1}' \subseteq T_{n+1}$.
\end{proof}

\begin{claim}
$p_{n+1} : (X \times \#N^{n+1}, T_{n+1}) -> (X, S_{n+1})$ is a trivial pseudofibration.
\end{claim}
\begin{proof}
Let $s \in S_{n+1}$; we must check that $T_{n+1}|p_{n+1}^{-1}(s)$ is contractible.  If $s \in S_n$ then clearly $T_{n+1}|p_{n+1}^{-1}(s) = T_{n+1}'|p_{n+1}^{-1}(s)$ so this follows from the previous claim that $p_{n+1} : T_{n+1}' -> S_n$ is a trivial pseudofibration.  So we may assume that $s \in S^{(n+1)}$, in which case
\begin{align*}
T_{n+1}|p_{n+1}^{-1}(s)
&= T_{n+1}'|p_{n+1}^{-1}(s) \cup \{s \times \{0\}^n \times \{c_{n+1}(s)\}\}.
\end{align*}
Since $p_{n+1} : T_{n+1}' -> S_n$ is a trivial pseudofibration, so is the restriction $p_{n+1} : T_{n+1}'|p_{n+1}^{-1}(s) -> S_n|s$; but this restriction has one-sided inverse the inclusion $S_n|s \cong S_n|s \times \{0\}^n \times \{c_{n+1}(s)\} \subseteq \@T_n(T_n|(s \times \#N^n)) \times \{c_{n+1}(s)\} \subseteq T_{n+1}'|p_{n+1}^{-1}(s)$, which is therefore a homotopy equivalence.  Now applying \cref{lm:pushout} to
\begin{align*}
T_{n+1}|p_{n+1}^{-1}(s) = T_{n+1}'|p_{n+1}^{-1}(s) \cup (S|s \times \{0\}^n \times \{c_{n+1}(s)\}),
\end{align*}
where the two subcomplexes on the right-hand side have intersection $S_n|s \times \{0\}^n \times \{c_{n+1}(s)\}$, yields that the inclusion $S|s \times \{0\}^n \times \{c_{n+1}(s)\} \subseteq T_{n+1}|p_{n+1}^{-1}(s)$ is a homotopy equivalence; but $S|s$ is a standard simplex, hence contractible, whence $T_{n+1}|p_{n+1}^{-1}(s)$ is contractible.
\end{proof}

This completes the definition of the complexes $T_n, T_n'$ and the verification that $p_n : T_n -> S_n$ is a homotopy equivalence for each $n$.  Note that from the definition and \cref{lm:telescope-props}(d), it is clear that each $T_n$ is $n$-dimensional.

\subsection{The constant bound}

We next bound the number of edges containing a point in $T_n$.  To do so, we will define for each $n \ge 1$ a constant $K_n$ such that for each $y \in X \times \#N^n$ there are at most $K_n$ distinct $y' \in X \times \#N^n$ with $y \le y'$ and $\{y, y'\} \in T_n$, and also the same holds with $y' \le y$.

For $n = 1$, we have $T_1' = T_0 \times N = S_0 \times N$, while $T_1 = T_1' \cup \{s \times \{c_1(s)\} \mid s \in S^{(1)}\}$.  Thus
\begin{align*}
K_1 := 3
\end{align*}
works: for $t = \{y \le y'\} \in T_1$, either $t \in T_1'$, in which case we have $y = (x, i)$ and $y' \in \{(x, i), (x, i+1)\}$ for some $(x, i) \in X \times \#N$, or $t = s \times \{c_1(s)\}$ for some $s \in S^{(1)}$, in which case $y = (x, c_1(s))$ and $y' = (x', c_1(s))$ for some $s = \{x < x'\} \in S^{(1)}$, which is uniquely determined by $y$ by the coloring property of $c_1$; and similarly for $y' \le y$.

Now suppose for $n \ge 1$ that we are given $K_n$; we find $K_{n+1}$ by a similar argument.  Let $t = \{y \le y'\} \in T_{n+1}$.  Since $n+1 \ge 2$, $T_{n+1}$ adds no $0$- or $1$-simplices to $T_{n+1}'$, so $t \in T_{n+1}'$.  If $t \in T_n \times N$, then we have $y = (z, i)$ and $y' = (z', i')$ for some $\{z \le z'\} \in T_n$ and $\{i \le i'\} \in N$, i.e., $i' \in \{i, i+1\}$; there are thus $\le 2K_n$ choices for $y'$ given $y$ in this case.  Otherwise, we have $t \in \@T_n(T_n|(s \times \#N^n)) \times \{c_{n+1}(s)\} \subseteq S|s \times N^n \times \{c_{n+1}(s)\}$ for some $s \in S^{(n+1)}$, whence $y = (x, i_1, \dotsc, i_n, c_{n+1}(s))$ and $y' = (x', i_1', \dotsc, i_n', c_{n+1}(s))$ where $x, x' \in s$ and each $i_j' \in \{i_j, i_j+1\}$; by the coloring property of $c_{n+1}(s)$, $s$ is uniquely determined by $y$, hence there are at most $|s| = n+2$ choices for $x'$ and so at most $(n+2)2^n$ choices for $y'$ given $y$.  In total, there are thus at most
\begin{align*}
K_{n+1} := 2K_n + (n+2)2^n
\end{align*}
choices for $y' \ge y$; similarly for $y' \le y$.

Solving this recurrence yields
\begin{align*}
K_n = 2^{n-2}(n^2+3n+2).
\end{align*}
So, for each $n \ge 1$ and $y \in X \times \#N^n$, there are at most $2(K_n-1)$ distinct edges $\{y < y'\}$ or $\{y' < y\}$ in $T_n$; that is, there are at most
\begin{align*}
M_n := 2(K_n-1) = 2^{n-1}(n^2+3n+2)-2
\end{align*}
edges in $T_n$ containing $y$.  When $S = S_n$ is $n$-dimensional, truncating the above inductive construction at $T_n$ and taking $T := T_n$ proves the combinatorial part of \cref{thm:lfsc-simp} (with the weaker condition ``at most $M_n$'' in (ii)) in this case.

\subsection{Growing edges}

Still in the $n$-dimensional case, in order to modify $T_n$ so that each vertex is contained in exactly $M_n$ edges, we use the following simple construction.  Put $T_{n,0} := T_n$.  Given $T_{n,k}$, let $T_{n,k+1}$ be $T_{n,k}$ together with, for each vertex $y$ of $T_n$ with fewer than $M_n$ edges, a new vertex $y'$ and an edge $\{y, y'\}$.  Then clearly
\begin{align*}
T_n^* := \bigcup_{k \in \#N} T_{n,k}
\end{align*}
is still $n$-dimensional and has each vertex contained in exactly $M_n$ edges.  Also, clearly $T_{n,k+1}$ deformation retracts onto $T_{n,k}$; thus (by \cref{lm:dirunion}) the inclusion $T_n = T_{n,0} \subseteq T_n^*$ is a homotopy equivalence.  So we may replace $T_n$ with $T_n^*$ to get the stronger form of \cref{thm:lfsc-simp}(ii).

\subsection{The infinite-dimensional case}

Next we handle the case where $S$ is infinite-dimensional.  Let $i_n : (X \times \#N^n, T_n) `-> (X \times \#N^{n+1}, T_{n+1})$ be the composite
\begin{align*}
i_n : T_n \cong T_n \times \{0\} \subseteq T_n \times N \subseteq T_{n+1}' \subseteq T_{n+1}.
\end{align*}
From the above diagram ($*$), we get a commutative diagram
\begin{equation*}
\begin{tikzcd}[column sep=4em]
T_0 \dar[->>,"p_0"',"\simeq"] \rar[hook,"i_0"] & T_1 \dar[->>,"p_1"',"\simeq"] \rar[hook,"i_1"] & T_2 \dar[->>,"p_2"',"\simeq"] \rar[hook,"i_2"] & \dotsb \\
S_0 \rar[hook] & S_1 \rar[hook] & S_2 \rar[hook] & \dotsb
\end{tikzcd}
\tag{$\dagger$}
\end{equation*}
We would like to let $T$ be the direct limit of the top row of this diagram, but that might not be locally finite.  Instead, we take the mapping telescope of the top row, which can be defined explicitly as follows.

Let $\#N^\infty$ be the direct limit of $\#N \cong \#N \times \{0\} \subseteq \#N^2 \cong \#N^2 \times \{0\} \subseteq \#N^3 \subseteq \dotsb$; explicitly, $\#N^\infty$ can be taken as the subset of $\#N^\#N$ consisting of the eventually zero sequences.  Then $X \times \#N^\infty$ is the direct limit of the sequence $X \times \#N^0 --->{i_0} X \times \#N^1 --->{i_1} \dotsb$, with injections
\begin{align*}
i^n : X \times \#N^n \cong X \times \#N^n \times \{0\}^\infty \subseteq X \times \#N^\infty;
\end{align*}
and so the direct limit of the top row of ($\dagger$) can be taken explicitly as the ordered simplicial complex $\bigcup_{n \in \#N} i^n(T_n)$ on $X \times \#N^\infty$.

The \defn{mapping telescope} of the top row of ($\dagger$) is the complex $(Y, T)$ where
\begin{align*}
Y &:= \bigcup_{n \in \#N} (X \times \#N^n \times \{0\}^\infty \times \{n, n+1\}) \subseteq X \times \#N^\infty \times \#N, \\
T &:= \bigcup_{n \in \#N} (i^n(T_n) \times N|\{n, n+1\}).
\end{align*}
For each $n$, let
\begin{align*}
\~T_n := \bigcup_{m \le n} (i^m(T_m) \times N|\{m, m+1\}).
\end{align*}
It is easy to see that the projection $p_1 : X \times \#N^\infty \times \#N -> X \times \#N^\infty$ restricts to simplicial maps $\~T_n -> i^n(T_n)$ for each $n$, yielding a commutative diagram
\begin{equation*}
\begin{tikzcd}[column sep=4em]
\~T_0 \dar[->>,"p_1"',"\simeq"] \rar[hook] & \~T_1 \dar[->>,"p_1"',"\simeq"] \rar[hook] & \~T_2 \dar[->>,"p_1"',"\simeq"] \rar[hook] & \dotsb \\
i^0(T_0) \rar[hook] & i^1(T_1) \rar[hook] & i^2(T_2) \rar[hook] & \dotsb
\end{tikzcd}
\tag{$\ddagger$}
\end{equation*}
in which the horizontal maps are inclusions and the vertical maps are homotopy equivalences by the usual argument: the (geometric realization of the) first cylinder $i^0(T_0) \times N|\{0, 1\}$ in $\~T_n$ deformation retracts onto its base $i^0(T_0) \times \{1\}$, which is contained in the second cylinder $i^1(T_1) \times N|\{1, 2\}$, which deformation retracts onto its base $i^1(T_1) \times \{2\}$, etc.  Since, as noted above, the bottom row of ($\ddagger$) may be identified with the top row of ($\dagger$), combining the two diagrams and applying \cref{lm:dirunion} yields that $T = \bigcup_n \~T_n$ is homotopy equivalent to $S = \bigcup_n S_n$ (via the restriction of the projection $X \times \#N^\infty \times \#N -> X$).

Since, clearly, each $T_n$ being locally finite implies that $T$ is locally finite, this proves the combinatorial part of \cref{thm:lfsc-simp} in the infinite-dimensional case.

\subsection{The Borel case}

Finally, suppose we start with a Borel structuring $S$ of a countable Borel equivalence relation $(X, E)$ by simplicial complexes.  Recall that this means $S$ is a simplicial complex on $X$ with simplices contained in $E$-classes and such that $S$ is Borel in the standard Borel space of finite subsets of $X$.  We may then simply apply the above construction to the locally countable simplicial complex $(X, S)$, while observing that each step is Borel.  To do so, we first pick a Borel linear order on $X$ to turn $(X, S)$ into an ordered simplicial complex, and then pick the coloring functions $c_n : S^{(n)} -> \#N$ to be Borel (in fact restrictions of a single $c : S -> \#N$) using the following standard lemma:

\begin{lemma}[{Kechris-Miller \cite[7.3]{KM}}]
\label{lm:coloring}
Let $(X, E)$ be a countable Borel equivalence relation, and let $[E]^{<\infty}$ be the standard Borel space of finite subsets of $X$ which are contained in some $E$-class.  Then there is a Borel $\#N$-coloring of the intersection graph on $[E]^{<\infty}$, i.e., a Borel map $c : [E]^{<\infty} -> \#N$ such that if $A, B \in [E]^{<\infty}$ with $A \ne B$ and $A \cap B \ne \emptyset$ then $c(A) \ne c(B)$.
\end{lemma}

It is now straightforward to check that the definitions of $T_n, T_n'$ are Borel; in the definition of $T_{n+1}'$, note that the union over $s \in S^{(n+1)}$ is disjoint, by the coloring property of $c_{n+1}$.  In the $n$-dimensional case, we end up with an ordered Borel simplicial complex $(X \times \#N^n, T_n)$ such that the projection $p_n : X \times \#N^n -> X$ is a homotopy equivalence $T_n -> S_n = S$.  Defining the countable Borel equivalence relation $F$ on $Y := X \times \#N^n$ by
\begin{align*}
(x, i_1, \dotsc, i_n) \mathrel{F} (x', i_1', \dotsc, i_n') \iff x \mathrel{E} x',
\end{align*}
we get that $T := T_n$ is a Borel structuring of $(Y, F)$; and we have a Borel embedding $f : (X, E) -> (Y, F)$ given by $f(x) := (x, 0, \dotsc, 0)$ such that $S|[x]_E$ is homotopy equivalent to $T|[f(x)]_F$ (via the map $p_n|([x]_E \times \#N^n) = p_n|[f(x)]_F : T|[f(x)]_F -> S|[x]_E$) for each $x \in X$.

For the stronger condition that each vertex is contained in exactly $M_n$ edges, it is straightforward that the definition of $T_n^*$ above can be taken to be a Borel simplicial complex on a standard Borel space $Y^* \supseteq Y$; letting $F^* \supseteq F$ be the obvious equivalence relation on $Y^*$ (so that each newly added edge in $T_n^*$ lies in one $F^*$-class), $T_n^*$ is a Borel structuring of $(Y^*, F^*)$ such that the composite $(X, E) --->{f} (Y, F) \subseteq (Y^*, F^*)$ is a homotopy equivalence on each class.  So we may replace $(Y, F, T_n)$ by $(Y^*, F^*, T_n^*)$.

Similarly, in the infinite-dimensional case, it is straightforward that the definition of the mapping telescope $T$ on $Y \subseteq X \times \#N^\infty \times \#N$ is Borel; so the same definitions of $F, f$ as in the finite-dimensional case work (note that $(x, 0, \dotsc, 0) \in Y$ for all $x \in X$).  This completes the proof of \cref{thm:lfsc-simp}, which implies \cref{thm:lfsc-cber}.

To prove \cref{thm:lfsc-compr}, apply \cref{thm:lfsc-cber} to get $(Y, F)$ with structuring $T$ and an embedding $f : (X, E) -> (Y, F)$; since $E$ is compressible, $f$ may be modified so that its image is $F$-invariant (see \cite[2.3]{DJK}), whence we get the desired structuring of $E$ by restricting $T$.

To prove \cref{cor:lfsc-cber}, let $S$ be the trivial structuring of $E$ given by $\{x_0, \dotsc, x_n\} \in S \iff x_0 \mathrel{E} \dotsb \mathrel{E} x_n$; this is obviously contractible on each $E$-class, so by \cref{thm:lfsc-simp} $E$ Borel embeds into some $F$ structurable by locally finite contractible complexes.  As before, this implies \cref{cor:lfsc-compr}.

\subsection{Some remarks}

In the dimension $n = 1$ case, the construction of $T_1$ above can be seen as a slight variant of the proof of Jackson-Kechris-Louveau \cite[3.10]{JKL}.  Thus the general case of our construction can be seen as a generalization of their proof to higher dimensions.

As mentioned in the Introduction, our construction is based on the proof of Whitehead \cite[Theorem~13]{Wh} that every countable CW-complex is homotopy equivalent to a locally finite complex of the same dimension.  That proof uses the same idea of ``spreading out'' cells along a ray to make their boundaries disjoint, but uses more abstract tools from homotopy theory in place of our explicit ``telescope'' construction $\@T_n$.  While it should be possible to give a more direct combinatorial transcription of Whitehead's proof, using (for example) simplicial sets, it does not seem that such an approach would yield a uniform bound $M_n$ on the number of edges containing a vertex in the $n$-dimensional case.

\section{Problems}
\label{sec:future}

There are several other nice properties of treeable countable Borel equivalence relations, for which we do not know if they generalize to higher dimensions.  Each of the following is known to be true in the case $n = 1$; see \cite[3.3, 3.12, 3.17]{JKL}.

\begin{problem}
Let $E, F$ be countable Borel equivalence relations such that $E$ Borel embeds into $F$.  If $F$ is structurable by $n$-dimensional contractible simplicial complexes, then must $E$ be also?
\end{problem}

\begin{problem}
Let $E$ be a countable Borel equivalence relation.  If $E$ is structurable by $n$-dimensional contractible simplicial complexes, then is $E$ necessarily structurable by $n$-dimensional locally finite contractible simplicial complexes?  (As noted in the Introduction, there cannot be a uniform bound on the number of edges containing each vertex.)
\end{problem}

\begin{problem}
Is there a single countably infinite $n$-dimensional contractible simplicial complex $S_n$, such that every countable Borel equivalence relation $E$ structurable by $n$-dimensional contractible simplicial complexes Borel embeds into an $F$ structurable by isomorphic copies of $S_n$?
\end{problem}

\begin{problem}
Is there a countable group $\Gamma_n$ with an $n$-dimensional Eilenberg-MacLane complex $K(\Gamma_n, 1)$, such that every countable Borel equivalence relation $E$ structurable by $n$-dimensional contractible simplicial complexes Borel embeds into the orbit equivalence relation of a free Borel action of $\Gamma_n$?
\end{problem}

\bigskip
\noindent Department of Mathematics

\noindent California Institute of Technology

\noindent Pasadena, CA 91125

\medskip
\noindent\nolinkurl{rchen2@caltech.edu}

\end{document}